\documentclass[a4paper,11pt,english]{amsart}

\usepackage{multirow}
\usepackage[english]{babel}
\usepackage[latin1]{inputenc}
\usepackage[T1]{fontenc}
\usepackage{wasysym}

\usepackage{graphicx}
\usepackage{hyperref}
\usepackage{amssymb}
\usepackage{amsmath}
\usepackage{verbatim,enumerate}
\usepackage{tikz}

\newtheorem{thm}{Theorem}[section]

\newtheorem{prop}[thm]{Proposition}
\newtheorem{cor}[thm]{Corollary}
\newtheorem{conj}[thm]{Conjecture}

{\theoremstyle{definition}
\newtheorem{exa}[thm]{Example}
\newtheorem{rem}[thm]{Remark}}
\newcommand{\C}{\mathbb{C}}
\newcommand{\T}{\mathbb{T}}
\newcommand{\CC}{\mathcal{C}}
\newcommand{\CP}{\mathbb{C}P}
\newcommand{\RP}{\mathbb{R}P}

\renewcommand{\epsilon}{\varepsilon}
\newcommand{\R}{\mathbb{R}}
\newcommand{\Z}{\mathbb{Z}}

\newcommand{\x}{\underline{x}}

\begin{document}

\title{On the invariance of  Welschinger invariants}

\author{Erwan Brugall\'e}
\address{Erwan Brugall\'e, Université de Nantes, Laboratoire de
  Mathématiques Jean Leray, 2 rue de la Houssinière, F-44322 Nantes Cedex 3,
France}
\email{erwan.brugalle@math.cnrs.fr}


\begin{abstract}
  We collect in this note some observations about 
  original Welschinger invariants  defined
 in \cite{Wel1}.
 None of their proofs is difficult, nevertheless these remarks do not seem
 to have been made before.
Our main result  is that when $X_\R$ is a real rational algebraic
surface, Welschinger invariants
  only depend on  the number of real
  interpolated points, and
some homological data associated to $X_\R$.
  This strengthened the invariance statement initially
proved by Welschinger. 

This main result follows easily from a formula relating Welschinger
invariants of two real symplectic manifolds differing by a 
surgery along a real Lagrangian sphere. In its turn, once one believes
that such formula may hold, its proof 
 is a mild adaptation of the proof of analogous formulas previously
 obtained by the author on the one hand, and by Itenberg, Kharlamov and
 Shustin on the other hand.

 We apply the two aforementioned results to
complete the computation of Welschinger invariants of
 real rational algebraic surfaces, and to obtain 
vanishing, sign, and sharpness results for these invariants that
generalize previously known statements.
We also discuss some hypothetical relations of our work with tropical
refined invariants defined in \cite{BlGo14} and \cite{GotSch16}.
\end{abstract}
\maketitle

\section{Main results}\label{sec:main}
A \textit{real symplectic manifold} $X_\R=(X,\omega_X,\tau_X)$ is a symplectic
manifold $(X,\omega_X)$ equipped with an anti-symplectic involution
$\tau_X$. The \textit{real part} of $(X,\omega_X,\tau_X)$, denoted by
$\R X$, is by
definition the fixed point set of $\tau_X$. An almost complex
structure $J$ on $X$ is called \emph{$\tau_X$-compatible} if it is
tamed by $\omega$, and if $\tau_X$ is $J$-anti-holomorphic.
In this note, the manifold  $X_\R$ will always be compact of 
dimension 4 with a non-empty real part, and we denote by
$H_2^{-\tau_X}(X;\Z)$ the space of $\tau_X$-anti-invariant classes. 
A non-singular
projective real algebraic variety  is always implicitly assumed to be
equipped with some
Kähler form which turns it into a real symplectic manifold.
All algebraic surfaces considered here are assumed to be projective
and non-singular.
\medskip

Let $X_\R=(X,\omega_X,\tau_X)$ be a real compact symplectic manifold of
dimension 4, and
 denote by 
$L_1,\cdots,L_k$ the connected components of $\R X$.
Choose a class
$d\in H_2(X;\Z)$, and
a vector $\rho=(r_1,\cdots,r_k)\in\Z_{\ge 0}^k$ such that
\[c_1(X)\cdot d - 1- \sum_{i=1}^kr_i=2s \in 2\Z_{\ge 0}.
\]
Choose a configuration $\x$ made of $r_i$ points in $L_i$ for
$i=1,\cdots k$, and  $s$ 
 pairs of $\tau_X$-conjugated 
points in $X\setminus \R X$.
Given a   $\tau_X$-compatible almost complex structure $J$,
we denote by $\mathcal C_{X_\R}(d,\x,J)$
 the set of real rational
$J$-holomorphic curves  in $X$ realizing the class $d$, and passing
 through $\x$. Then we define the integer
 \[
 W_{X_\R,\rho}(d;s)= \sum_{C\in\mathcal C_{X_\R}(d,\x,J)}(-1)^{m(C)},
 \]
where
$m(C)$ is the number of nodes of  $C$ in $\R X$ with two
$\tau_X$-conjugated branches.
For a generic choice of $J$, the set  $\mathcal C_{X_\R}(d,\x,J)$ is finite, and
 $W_{X_\R,\rho}(d;s)$
depends neither on $\x$, $J$, nor on the deformation class of
 $X_\R$ (see \cite{Wel1,Wel11}).
 We call these numbers the \emph{Welschinger invariants of $X_\R$}.
 The main result of this paper, Theorem \ref{thm:main0} below,
 is that when $X_\R$ is a real rational
 algebraic surface,
 Welschinger invariants eventually only depends on $s$ and some
 homological data of $X_\R$.
 
 \begin{rem}\label{rem:vanish}
The set $\mathcal C_{X_\R}(d,\x,J)$ is clearly empty when either
$c_1(X)\cdot d \le 0$, or
$d\notin H_2^{-\tau_X}(X;\Z)$, or the partition
$\rho$ contains two positive elements. This implies that
$W_{X_\R,\rho}(d;s)=0$ in these cases.
 \end{rem}
 \begin{rem}
   We
     restrict ourselves  to the original Welschinger invariants as defined in
     \cite{Wel1}, and do not consider the more general modified Welschinger invariants
     defined in \cite{IKS11,IKS13,IKS14}, and also
     considered in \cite{BP14,Bru16}. Further,
     because of Remark \ref{rem:vanish}, 
     the invariants $W_{X_\R,\rho}(d;s)$ are usually considered with
    a partition $\rho$ containing a single positive entry
    corresponding to a connected component $L$ of $\R X$. In this
    case, the invariant
     $W_{X_\R,\rho}(d;s)$ in this text corresponds to the invariant
$W_{X_\R,L,[\R X\setminus L]}(d;s)$ in \cite{BP14,Bru16}, to the invariant
     $W_s(X_\R,d,L,0)$ in \cite{IKS14}, and to the invariant
     $W(X_\R,d,L,0)$ in \cite{IKS11,IKS13} if $s=0$. Our
    motivation to consider 
    arbitrary partitions $\rho$ comes from Theorem \ref{thm:main0}.
    \end{rem}

 Two real rational algebraic surfaces $X_{1,\R}$ and $X_{2,\R}$
 are said to be
 \emph{homologically equivalent} if both are obtained, up to deformation, as
 a real blow-up $\pi_i:X_{i,\R}\to X_{0,\R}$ of a real minimal algebraic surface
$X_{0,\R}$ at  $p$ distinct real points  and   
$q$ distinct pairs of $\tau_{X_0}$-conjugated points. We emphasize
 that the distributions
 of the $p$ real points  among 
 connected components of  $\R X_{0}$ may not coincide for $\pi_1$ and $\pi_2$.
 Note nevertheless that
 \[
 \chi(\R X_{1})=\chi(\R X_{2})=\chi(\R X_0)-p.
 \]
Denoting by $E_1,\cdots E_p$ (resp.
$F_1,\overline F_1,\cdots,F_q,\overline F_q$) the exceptional divisors
coming from the blow-ups at real points (resp. at pairs of
$\tau_{X_0}$-conjugated points), the map $\pi_i$ induces the  following 
decomposition 
\[
H_2(X_i;\Z)= H_2(X_0;\Z) \oplus \bigoplus_{j=1}^p
\Z[E_j]\oplus\bigoplus_{j=1}^q (\Z[F_j]\oplus
\Z[\overline F_j])
\]
which is orthogonal with respect to the intersection form.
Furthermore,  the action of $\tau_{X_i}$ is given by
\[
\tau_{X_i,*}|_{H_2(X_0;\Z)}=\tau_{X_0,*}\qquad
\tau_{X_i,*}([E_j])=-[E_j]  \mbox{ for }j=1,\cdots,p\qquad
\tau_{X_i,*}([F_j])=-[\overline F_j]  \mbox{ for }j=1,\cdots,q.
\]
In particular, the two maps $\pi_1$ and $\pi_2$ provide an identification of the
groups $H_2(X_1;\Z)$ and $H_2(X_2;\Z)$ commuting with  both 
intersection forms and   action of the anti-symplectic involutions.
 We denote by $[X_\R]$ the
homological equivalence class of a  real rational algebraic surface $X_\R$.

\begin{thm}\label{thm:main0}
If   $X_\R$ is a 
real  rational algebraic surface, then 
 $W_{X_\R,\rho}(d;s)$ does
not depend on $\rho$, nor on a particular representative of
$[X_\R]$.
\end{thm}
As a consequence of Theorem  \ref{thm:main0},  we simply denote by 
$W_{[X_\R]}(d;s)$ the invariant  $W_{X_\R,\rho}(d;s)$.
Note that some particular instances of Theorem \ref{thm:main0} have
already been noticed 
in  \cite[Corollary 6.11 and Theorem 7.5]{Bru14} and \cite[Corollary 4.5]{Bru16}.
Further, it
follows from the classification of  real rational algebraic surfaces up
to deformation established in
\cite[Main Theorem and Theorem 2.4.1]{DegKha02}
that all Welschinger
invariants of projective real rational algebraic surfaces are
determined by the following ones\footnote{Recall that all real algebraic
  surfaces considered here are assumed to have a non-empty real part.}:
\begin{itemize}
\item  $W_{[\CP^2_{p,q}]}(d;s)$, where $\CP^2_{p,q}$ is a real
  blow-up  of $\CP^2$ in $p$ real points and $q$ pairs of complex conjugated
  points; 
\item $W_{[QH_{0,q}]}(d;s)$, where $QH_{0,q}$ is the real
  blow-up  of
  the quadric hyperboloid in $\CP^3$ in  $q$ pairs of complex conjugated
  points;

\item  $W_{[QE_{0,q}]}(d;s)$, where $QE_{0,q}$ is a real
  blow-up  of
  the quadric ellipsoid in $\CP^3$ in  $q$ pairs of complex conjugated
  points;

\item $W_{[CBN_{p,q}]}(d;s)$, where $CBN_{p,q}$ is a real
  blow-up in  $p$ real points and $q$ pairs of complex conjugated
  points of a real minimal conic bundle whose real part consists
  in the disjoint union of $N\ge 2$ spheres $S^2$;

\item  $W_{[DP2_{p,q}]}(d;s)$, where $DP2_{p,q}$ is a real
  blow-up in  $p$ real points and $q$ pairs of complex conjugated
  points of a real minimal del Pezzo surface of degree 2 whose real part consists
  in the disjoint union of $4$ spheres $S^2$;
\item $W_{[DP1_{p,q}]}(d;s)$, where $DP1_{p,q}$ is a real
  blow-up in  $p$ real points and $q$ pairs of complex conjugated
  points of a real minimal del Pezzo surface of degree 1 whose real part consists
  in the disjoint union of $\RP^2$ and $4$ spheres $S^2$.

\end{itemize}

 \begin{rem}
 Loosely speaking, 
 Theorem \ref{thm:main0}  states that $W_{X_\R,\rho}(d;s)$ only
 depends on $s$ and the lattice $H_2(X;\Z)$  equipped with the
 intersection form and the action of $\tau_{X,*}$. It may be
 interesting to work this out more rigorously. It may  also be
 interesting to study generalizations of Theorem \ref{thm:main0} to modified
 Welschinger invariants introduced in \cite{IKS11}, as well as to to  higher
 genus Welschinger invariants introduced in \cite{Shu14}, or to the
higher dimensional invariants recently defined in \cite{Geor13,GeoZin15}.
 \end{rem}

  Theorem
 \ref{thm:main0} easily implies next corollary, which
 generalizes \cite[Theorem
   1.1(1)]{BP14} in the case $F=[\R X_\R\setminus L]$.
 \begin{cor} \label{cor:maincor}
Let $X_\R$ be a compact
real  rational algebraic surface with a disconnected real
part. Suppose that $X_\R$ is a real blow-up of another real  rational
algebraic surface in at least two real points, and denote by $E_1$ and
$E_2$ the corresponding exceptional divisors.
Then for any $d\in H_2(X;\Z)$ such that both $d\cdot [E_1]$ and
$d\cdot [E_2]$ are odd, one has $W_{[X_\R]}(d;s)=0$.
 \end{cor}

 Combining Theorem \ref{thm:main0} with
 \cite[Theorem 1.1]{Wel4} and Corollary \ref{cor:maincor}, we obtain the
following.
\begin{thm}\label{thm:sign}
Let $X_\R$ be a compact
real  rational algebraic surface with a disconnected real
part, and assume that
 $ c_1(X)\cdot d -1-2s>0$. Then
 one has
 \[
 (-1)^{\frac{d^2-c_1(X)\cdot d+2}{2}}\cdot
W_{[X_\R]}(d;s)\ge 0. 
\]
Furthermore, the invariant $W_{[X_\R]}(d;s)$
is sharp in the following sense: there exists a compact
real  rational algebraic surface $Y_\R$ in $[X_\R]$,  a
real configuration $\x$ of $c_1(X)\cdot d -1$ points in $Y$ with
$|\x \cap \R Y|=c_1(Y)\cdot d-2s$, and a generic $\tau_Y$-compatible
almost complex structure $J$ on $Y$ such that
\[
Card(\CC_{Y_\R}(d,\x,J)) = |W_{[X_\R]}(d;s)|.
\]
\end{thm}

\begin{rem}
  A configuration $\x$ and a $\tau_Y$-compatible
almost complex structure $J$ as in Theorem \ref{thm:sign}
may not exist for any
representative $Y_\R$ of $[X_\R]$, even up to deformation, see \cite[Remark 6.13]{Bru14}.
\end{rem}

One of the main ingredients in our proof of Theorem \ref{thm:main0} is Theorem
\ref{thm:main} that relates Welschinger invariants of two  real
symplectic 4-manifolds differing by a so-called
\emph{surgery along a real Lagrangian sphere}. We refer to
Section \ref{sec:surgery} for more details about this operation, and
for a statement of Theorem \ref{thm:main} together with its proof. 
This latter theorem  partially generalizes both \cite[Corollary 4.2]{IKS13} and
\cite[Theorem 1.1, Remark 1.3]{Bru16}. We point out that
its proof is an easy adaptation
of the proof of  \cite[Corollary 4.2]{IKS13}, using
\cite[Theorem 2.5(1)]{BP14}. It just required 
to believe in the correctness of the statement to prove it.  

Combining \cite[Main Theorem and Theorem 2.4.1]{DegKha02}
together with  the classical rigid isotopy classifications
of plane real quartics, of real cubic sections of the quadratic cone
in $\CP^3$, and of real quadrics in $\CP^3$ (see for example
\cite{DK}), one easily classifies real rational
algebraic surfaces 
 up to  deformation,
real blow-up, and surgery along a real
Lagrangian sphere: any real rational surface is obtained by a finite
sequence of these three operations starting from either 
$\CP^2$ or the quadric hyperboloid $QH$.
In particular we have the following result.
\begin{thm}\label{thm:all W}
  Let $X_\R$ be a real algebraic rational surface. Then
 by  finitely many successive applications of Theorem
\ref{thm:main}, all Welschinger
invariants of $X_\R$ can be computed out of  
Welschinger
invariants of either $\CP^2_{p,0}$ or $QH$.
\end{thm}
Since all Welschinger invariants of $QH$ and  $CP^2_{p,0}$  have been
computed, see for example \cite{Mik1,Br6b,HorSol12,Che18}, Theorem \ref{thm:all W} completes
the computation of Welschinger invariants $W_{[X_\R]}(d;s)$ of
 real
rational algebraic surfaces.

Next statement can be seen as an increasingness property of
$W_{[X_\R]}$ with respect to $\chi(\R X_\R)$, and goes in a somewhat different
direction than \cite[Theorem 3.4]{Br20}, \cite[Proposition 2.8]{BP14},
and \cite[Corollaries 4.4 and 6.10]{Bru14}.
\begin{thm}\label{thm:increase}
  Let  $X_\R$ and $Y_\R$ be two compact real rational algebraic
  surfaces with a disconnected real part, differing by 
  a surgery  along a real Lagrangian sphere, and such that
  $\chi(\R Y)=\chi(\R X)+ 2$.
  Then for any  $d\in H^{-\tau_Y}_2(Y;\Z)$ and $s\in\Z_{\ge 0}$ such that
  $c_1(X)\cdot d -1 -2s>0$, one has
 \[
  |W_{[Y_\R]}(d;s)| \ge   |W_{[X_\R]}(d;s)|.
\]
\end{thm}

Theorem \ref{thm:main} is obtained thanks to a real version of a (very
simple instance) of the symplectic sum formula from
\cite{IP,LiRu01,TehZin14},
see also\cite{Li02,Li04}
for an analogous formula in the complex algebraic category.
It turns out that the same strategy provides
a formula similar to  Theorem
\ref{thm:main} for
relative Gromov-Witten invariants of symplectic
4-manifolds.
This observation suggests a possible connection of our work to
tropical refined invariants defined in \cite{BlGo14,GotSch16}.
We discuss this aspect in Section \ref{sec:refined}. In particular, we
provide there
an alternative explanation for the specializations in $q=\pm 1$ of the 
tropical refined descendant invariants from \cite{GotSch16}. We also show
that
a refined version of a conjecture by Itenberg, Kharlamov and Shustin
\cite[Conjecture 6]{IKS2} holds, although it was known to be wrong in
the non-refined case.

\medskip
The remaining part of the
paper is organised as follows. We introduce surgeries along real
Lagrangian spheres in Section \ref{sec:surgery},  then state and
prove Theorem \ref{thm:main}.
All statements given
in the present section can easily be derived  from Theorem \ref{thm:main} and
previously known results. Proofs are
given in 
 Section \ref{sec:proof}.
Finally, we
discuss in Section \ref{sec:refined} connections
of our work to
tropical
refined invariants
of algebraic surfaces and refined Severi degrees.

\medskip
\noindent{\bf Acknowledgment.}  This  work is
partially supported by the grant TROPICOUNT of Région Pays de la
Loire, and the ANR project ENUMGEOM NR-18-CE40-0009-02.
I thank Arthur Renaudineau for his comments on the first version of
this paper.

\section{Surgery along a real Lagrangian sphere}\label{sec:surgery}
Let $X_\R=(X,\omega_X,\tau_X)$ be a real compact symplectic manifold of
dimension 4, and let
 $S\subset X$ be a Lagrangian sphere
 globally invariant under $\tau_X$.
 Locally, a neighborhood $V$ of $S$
 is given by a neighborhood in
  the real affine quadric $(Q,\omega_Q, \tau)$ in $\C^3$ given
by the equation
$$(-1)^{\epsilon_1} x^2+(-1)^{\epsilon_2} y^2+(-1)^{\epsilon_3} z^2=1
\qquad \mbox{with }\epsilon_i\in\{0,1\},$$
 of the sphere $S_Q$ in
$i^{\epsilon_1}\R\times i^{\epsilon_2}\R \times i^{\epsilon_3}\R$
with equation
$$x^2+y^2+z^2=1.$$
As explained in \cite[Section 2.2]{Bru16},
one can modify the symplectic and real structure of  $X_\R$ in $V$ so
that $V$ is now given 
by
a neighborhood of $S_Q$ 
in the
real affine quadric  in $\C^3$ with equation
$$(-1)^{\epsilon_1} x^2+(-1)^{\epsilon_2} y^2+(-1)^{\epsilon_3} z^2=-1 .$$  
The resulting real symplectic manifold 
$Y_\R$ is called a \emph{surgery of $X_\R$ along $S$}  (see Figure
\ref{fig:degen} for a local picture). 
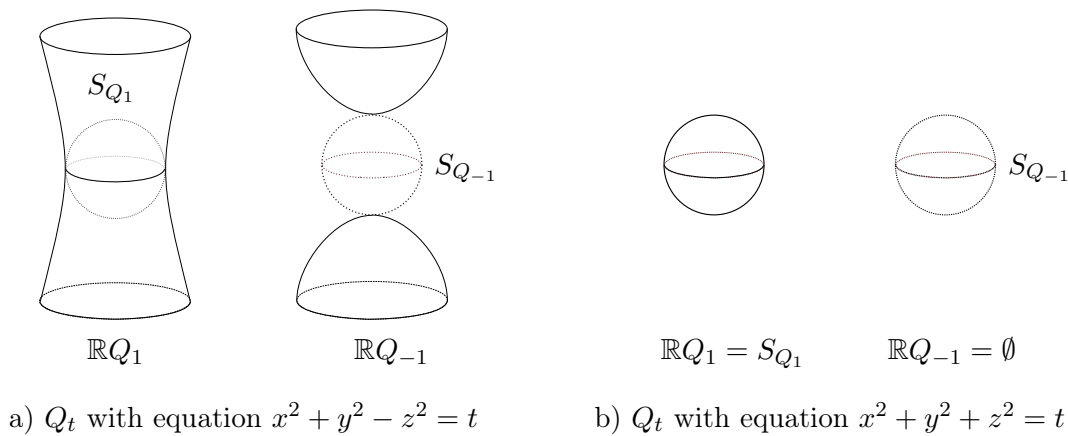
\begin{figure}[h!]
\begin{center}
\begin{tabular}{ccc}
\includegraphics[width=2cm, angle=0]{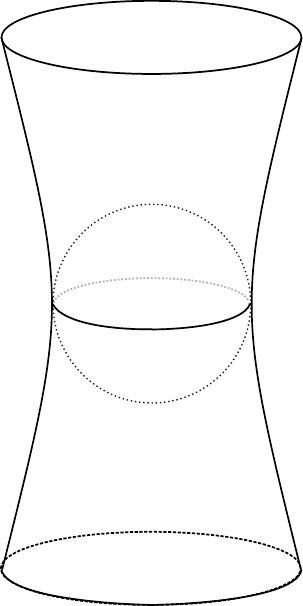}
\hspace{3ex} 
\hspace{3ex} 
\includegraphics[width=2cm, angle=0]{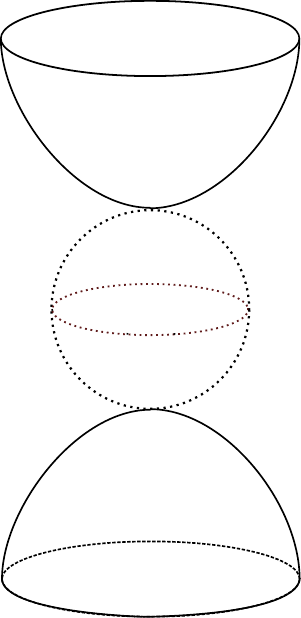}
\put(-135,-15){$\R Q_{1}$}
\put(-135,85){$S_{Q_1}$}
\put(-35,-15){$\R Q_{-1}$}
\put(-5,55){$S_{Q_{-1}}$}
& \hspace{5ex}
&\includegraphics[width=2cm, angle=0]{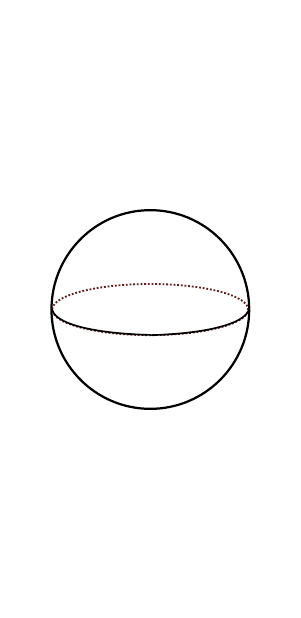}
\hspace{3ex} 
\hspace{1ex} 
\includegraphics[width=2cm, angle=0]{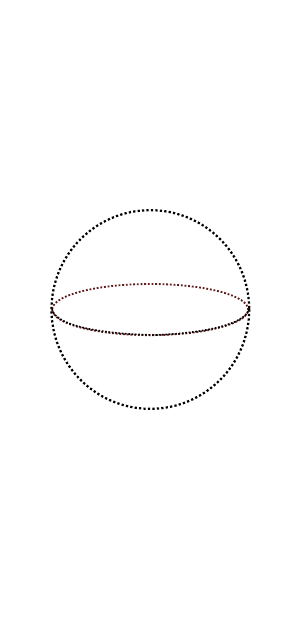}
\put(-135,-15){$\R Q_{1}=S_{Q_1}$}
\put(-50,-15){$\R Q_{-1}=\emptyset$}
\put(-5,55){$S_{Q_{-1}}$}
\\ \\a) $Q_t$ with equation $x^2+ y^2- z^2=t$
&&b) $Q_t$ with equation $x^2+ y^2+ z^2=t$
\end{tabular}
\end{center}
\caption{Surgery of a real symplectic 4-manifold along a real Lagrangian sphere.}
\label{fig:degen}
\end{figure}
Note that, with the convention that $\chi(\emptyset)=0$, we have
$$\chi(\R Y)=\chi(\R X)\pm 2.$$
Furthermore  the class $[S]$ in $H_2(X;\Z)$ is $\tau_X$-anti-invariant if and
only if it is $\tau_Y$-invariant (in which case we have 
$\chi(\R Y)=\chi(\R X)+ 2$).

\medskip
Suppose that  $\chi(\R Y)=\chi(\R X)+ 2$, and let $\rho$ be a vector
whose
entries are indexed by  connected components of $\R Y$. If
$S\subset \R Y$, we further assume that the entry of $\rho$
corresponding to $S$ vanishes. 
We denote by $S^{\tau_X}$ (resp. $S^{\tau_Y}$)
the fixed point set of the involution
$\tau_X|_S$ (resp.  $\tau_Y|_S$). In
particular, we have either $S^{\tau_X}=S^1$ and $S^{\tau_Y}$ consists
in 2 points, or $S^{\tau_X}=\emptyset$ and $S^{\tau_Y}=S$.
The connected components of $\R Y\setminus S^{\tau_Y}$
are canonically  in one-to-one correspondence with the connected
components of $\R X\setminus S^{\tau_X}$. Hence  one can associate to
$\rho$ 
a vector $\widetilde \rho$ whose entries are indexed by
the connected components of $\R X$: the entry corresponding to
the connected component
$L$ of $\R X$ is the sum of the entries corresponding to
the connected component of
$\R Y\setminus S^{\tau_Y}$ corresponding to
$L\setminus  S^{\tau_X}$.

Next theorem  partially generalizes both \cite[Corollary 4.2]{IKS13} and
\cite[Theorem 1.1, Remark 1.3]{Bru16}. 

\begin{thm}\label{thm:main} 
Let  $X_\R$ be a compact real symplectic manifold of dimension 4. 
Let $S$ be a real Lagrangian sphere in $X_\R$, and
 $Y_\R$ a  surgery of $X_\R$ along
$S$. Suppose that  $\chi(\R Y)=\chi(\R X)+ 2$.
Then  for any class 
$d\in H^{-\tau_Y}_2(X;\Z)$,
the  following identity holds
\[
  W_{Y_\R,\rho}(d;s) = \sum_{k\in \Z}\ (-1)^k\ W_{X_\R,\widetilde \rho}(d-k[S];s),
  \]
  whenever the entry of $\rho$ corresponding to $S$ vanishes if
  $S\subset \R Y$.
\end{thm}
\begin{proof}
  Note first that the identity we want to prove is independent on the
  orientation chosen on $S$ to define an element $[S]\in H_2(X;\Z)$.
 By \cite[Remark 1.3]{Bru16},   this identity
   can be rewritten in the following form once such orientation is chosen:
  \begin{equation}\label{eq:proof}
  W_{Y_\R,\rho}(d;s) = W_{X_\R,\widetilde \rho}(d;s) + 2\sum_{k\ge
      1}\ (-1)^k\
    W_{X_\R,\widetilde \rho}(d-k[S];s).
  \end{equation}
  Recall (see \cite[Section 2.2]{Bru16}) that $X_\R$ can be deformed to 
a real symplectic 4-manifold $(Z,\omega_Z,\tau_Z)$ for which $S$,
equipped with the chosen orientation, 
becomes symplectic.
Choose a real configuration of points $\x$ in $Z$ with $s$ pairs of
$\tau_Z$-conjugated points, and $\rho_L$ points in $L$ for
each connected component $L$ of $\R Z\setminus \R S$. Choose also a
$\tau_Z$-compatible almost complex structure $J$ on $Z$ for which $S$
is $J$-holomorphic.
Given an integer $k\ge 0$, we denote 
 by $\mathcal C^\beta(d-k[S],\underline x,J)$
the set of all irreducible rational real $J$-holomorphic curves in
$(Z,\omega_Z,\tau_Z)$ 
  passing through all points in $\x$, realizing the class
  $d-k[S]$, and intersecting $S\setminus \R S$ in exactly $\beta$ pairs of
  $\tau_Z$-conjugated points. For a generic choice of $J$ satisfying the above conditions, 
the set $\mathcal C^\beta(d-k[S],\underline x,J)$ is finite and only contains
  nodal curves by \cite[Lemma 3.1 and Proposition 3.3]{BP14}. We define
  \[
  W^\beta_{Z_\R,\widetilde \rho}(d-k[S];s)=\sum_{C\in \mathcal C^\beta(d-k[S],\underline x,J)}(-1)^{m(C)}.
  \]
Note that $W^\beta_{Z_\R,\widetilde \rho}(d-k[S];s)$ may depend on
the choices of $\x$ and $J$, nevertheless we will not record this
dependence in our notation in order to lighten the exposition.
  By \cite[Theorem 2.5(1)]{BP14}, we have
 \[
  W_{Y_\R,\rho}(d;s) = \sum_{l\ge 0}(-2)^l \ W^l_{Z_\R,\widetilde \rho}(d-l[S];s),
  \]
  and
   \[
   W_{X_\R,\widetilde \rho}(d;s) =  \sum_{j,b,\beta\ge
     0}
   \left(\begin{array}{c}d\cdot[S]+2j-2b \\ j-2\beta\end{array}\right)
     \left(\begin{array}{c}b \\ \beta\end{array}\right)
  W^b_{Z_\R,\widetilde\rho}(d-j[S];s).
  \]
Denoting by $A$ the right-hand-side of
$(\ref{eq:proof})$, and using the last identity, we obtain
\begin{align*}
A &=& &\sum_{j,b,\beta\ge
     0}
   \left(\begin{array}{c}2j-2b \\ j-2\beta\end{array}\right)
     \left(\begin{array}{c}b \\ \beta\end{array}\right)
       W^b_{Z_\R,\widetilde \rho}(d-j[S];s)
       \\
       \\ & && +2 \sum_{k\ge 1}\ (-1)^k
      \sum_{j,b,\beta\ge 0}
   \left(\begin{array}{c}2k+2j-2b \\ j-2\beta\end{array}\right)
     \left(\begin{array}{c}b \\ \beta\end{array}\right)
       W^b_{Z_\R,\widetilde \rho}(d-(j+k)[S];s).       
\end{align*}
By the change of variable $j=l-k$,
we obtain
\begin{align*}
A &=& &\sum_{l,b,\beta\ge
     0}
   \left(\begin{array}{c}2l-2b \\ l-2\beta\end{array}\right)
     \left(\begin{array}{c}b \\ \beta\end{array}\right)
       W^b_{Z_\R,\widetilde \rho}(d-l[S];s)
       \\
       \\ & && +2\sum_{k,l\ge 1}
      \sum_{b,\beta\ge 0} (-1)^k
   \left(\begin{array}{c}2l-2b \\l-k-2\beta\end{array}\right)
     \left(\begin{array}{c}b \\ \beta\end{array}\right)
       W^b_{Z_\R,\widetilde\rho}(d-l[S];s).
\end{align*}
Hence the coefficient of
$W^b_{Z_\R,\widetilde \rho}(d-l[S];s)$ in $A$ is
\[
\sum_{\beta\ge
     0}
   \left(\begin{array}{c}2l-2b \\ l-2\beta\end{array}\right)
     \left(\begin{array}{c}b \\ \beta\end{array}\right) +
       2\sum_{k\ge 1}
      \sum_{\beta\ge 0} (-1)^k
   \left(\begin{array}{c}2l-2b \\ l-k-2\beta\end{array}\right)
     \left(\begin{array}{c}b \\ \beta\end{array}\right),
\]
that is to say
\[
\sum_{\beta\ge
     0}\left(
   \left(\begin{array}{c}2l-2b \\ l-2\beta\end{array}\right)
     +  2\sum_{k\ge 1}
      (-1)^k
   \left(\begin{array}{c}2l-2b \\ l-k-2\beta\end{array}\right)       \right)
     \left(\begin{array}{c}b \\ \beta\end{array}\right).
\]
We denote by $u_{l,b,\beta}$ the coefficient of
$\left(\begin{array}{c}b \\ \beta\end{array}\right)$ in the latter
  sum. We have
  \begin{align*}
 u_{l,b,b-\beta} & = \left(\begin{array}{c}2l-2b \\ l-2\beta\end{array}\right)
     +  2\sum_{k\ge 1}
      (-1)^k
   \left(\begin{array}{c}2l-2b \\ l+k-2\beta\end{array}\right).
  \end{align*}
  Hence we get
  \[
  u_{l,b,\beta} + u_{l,b,b-\beta} =2\times (-1)^{l}\  \sum_{p\ge 0}^{2l-2b}
      (-1)^p   \left(\begin{array}{c}2l-2b \\ p\end{array}\right),
 \]
 and so
 \[
 u_{l,b,\beta} + u_{l,b,b-\beta} = 0
 \]
 if $l>b$, and
  \[
 u_{l,l,\beta} + u_{l,l,l-\beta} = 2 \times (-1)^l.
 \]
 This implies that
 \[
 \sum_{\beta\ge
     0} u_{l,b,\beta}
     \left(\begin{array}{c}b \\ \beta\end{array}\right)=0
 \]
 if $l>b$, and
 \[
 \sum_{\beta\ge
     0} u_{l,l,\beta}
     \left(\begin{array}{c}l \\ \beta\end{array}\right)=(-2)^l,
       \]
       which is precisely what we have to prove.
\end{proof}

 \begin{rem}
   Theorem \ref{thm:main} can clearly be generalised to
   modified Welschinger invariants introduced in
   \cite{IKS11}, at the cost of much heavier notations. It may be
   nevertheless interesting to work out such generalisation.
 \end{rem}

As an example of application of Theorem \ref{thm:main}, we have the following.
\begin{prop}\label{prop:wel}
  Let  $X_\R$ be a compact real symplectic manifold of dimension
  4. Let $\widetilde X_\R$ be the blow up of $X_\R$ at two disjoint
  real balls, and $\widetilde Y_\R$ be the blow up of $X_\R$ at two 
  $\tau_X$-conjugated disjoint balls. In both cases we denote by $E_1$
  and $E_2$ the two exceptional divisors.
  Then for any class
  $d\in  H^{-\tau_X}(X;\Z)$ and any $l\in\Z_{\ge 0}$, one has
  \[
  W_{\widetilde Y_\R,\rho}(d-l [E_1] - l[E_2];s)=W_{\widetilde X_\R,\rho}(d-l [E_1] - l[E_2];s)
  +2\sum_{\lambda=1}^l(-1)^\lambda\ W_{\widetilde
    X_\R,\rho}(d-(l-\lambda)[E_1] - (l+\lambda)[E_2] ;s). 
  \]
\end{prop}
The notation $\rho$ both for $\widetilde X_\R$ and
$\widetilde Y_\R$ in the theorem makes sense since the connected components of
$\R\widetilde X$ and $\R\widetilde Y$ are in a canonical one-to-one
correspondence. 
\begin{proof}
This is Theorem \ref{thm:main} applied to the class $[E_1]-[E_2]$.
\end{proof}

Proposition \ref{prop:wel} applied with $l=1$, combined with
\cite[Theorems 1.1 and 1.2]{DiHu18},  specializes to Welschinger Formula
\cite[Theorem 0.4]{Wel1}.
\begin{cor}\label{cor:wel}
  Let  $X_\R$ be a compact real symplectic manifold of dimension
  4, and let $\widetilde X_\R$ be the blow up of $X_\R$ at one real ball. We denote by $E$ the exceptional divisor.
 Then for any $d\in H^{-\tau_X}(X;\Z)$ and $s\in\Z_{\ge 0}$ such that
  $c_1(X)\cdot  d-2s\ge 3$, one has
  \[
  W_{X_\R,\rho}(d;s+1)=W_{X_\R,\rho}(d;s) - 2W_{\widetilde X_\R,\rho}(d-2E;s).
  \]
\end{cor}

\section{Proofs of the main results}\label{sec:proof}
As mentioned in Section \ref{sec:main},
a classification of real rational
algebraic surfaces 
 up to  deformation,
real blow-up, and surgery along a real
Lagrangian sphere
is easily obtained combining \cite[Main Theorem and Theorem 2.4.1]{DegKha02}
together with  the classical rigid isotopy classifications
of plane real quartics, of real cubic sections of the quadratic cone
in $\CP^3$, and of real quadrics in $\CP^3$ (see for example
\cite{DK}).
We implicitly use this classification in the following five proofs.

\begin{proof}[Proof of Theorem \ref{thm:main0}]
  The surface $X_\R$ can
  be degenerated 
to 
   a nodal  real algebraic
  rational surface $\overline X_\R$
  with a connected real part and having only real
  nodes.
  Blowing up these
  nodes, we obtain a non-singular real algebraic rational surface
  $Z_\R$ with a connected real part.
  By construction  $X_\R$ is obtained, up to deformation,
  by surgeries along the disjoint union of the exceptional divisors
  of the desingularization $Z_\R\to\overline X_\R$.
  Since $\R Z$ is connected, the Welschinger invariants
  of $Z_\R$ depend neither on the choice of $\x$ nor on
  the position of the real blown-up points on a minimal model of $Z$.
Now   the proposition is  an immediate consequence of
  Theorem \ref{thm:main}.
\end{proof}

\begin{proof}[Proof of Corollary \ref{cor:maincor}]
Up to choosing another representative of $[X_\R]$, we may assume that
$\R E_1$ and $\R E_2$ do not lye on the same connected component of
$\R X$. In particular by connectedness of $\RP^1$,
the set $\mathcal C_{X_\R}(d,\x,J)$
is clearly empty for any configuration $\x$.
\end{proof}

\begin{proof}[Proof of Theorem \ref{thm:sign}]
 Up to choosing  another
  representative of $[X_\R]$,  we may assume that a connected
  component $L$ of $\R X$ is homeomorphic to the sphere $S^2$.
If $c_1(X)\cdot d-1-2s \ge 2$, the result follows from Remark
\ref{rem:vanish}.
If  $c_1(X)\cdot d-1-2s =1$, then we choose the configuration $\x$
such that $\x\cap \R X\subset L$. Now the result follows from
\cite[Theorem 1.1]{Wel4}. 
\end{proof}
\begin{rem}
If  $c_1(X)\cdot d-1-2s =0$, then \cite[Theorem 1.1]{Wel4} states that
one can find $\x$ and $J$ so that there are no curve $C$ in
$\mathcal C_{X_\R}(d,\x,J)$ with $\R C\subset L$. Nevertheless, 
there may exist curves $C$ in $\mathcal C_{X_\R}(d,\x,J)$ with
$\R C\subset \R X\setminus L$.
\end{rem}

\begin{proof}[Proof of Theorem \ref{thm:all W}]
 Any real rational algebraic surface can be obtained
by a finite sequence of deformations,
real blow-ups, and surgery along real
Lagrangian spheres,
starting from either  $\CP^2$ or the quadric hyperboloid $QH$ in $\CP^3$.

Let $Y_{1,\R}$ and $Y_{2,\R}$ be two real algebraic surfaces obtained
by blowing up a real rational surface $Y_\R$ in respectively two real
and $\tau_Y$-conjugated points. Denoting by $E_1$ and $E_2$ the two
exceptional divisors, the real surface  $Y_{2,\R}$ is obtained by a
surgery along a real Lagrangian sphere realizing the class $[E_1]-[E_2]$. Since
$QH$ blown-up in $l\ge 1$ real points is also $\CP^2$ blown-up in $l+1$
real points, the result follows.
\end{proof}

\begin{proof}[Proof of Theorem \ref{thm:increase}]
Let $S$ be a real  Lagrangian sphere allowing to pass from $X_\R$ to
$Y_\R$ by surgery. Since one has $[S]^2=-2$ and $c_1(X)\cdot [S]=0$,
it follows from Theorem \ref{thm:sign} that
$W_{[X_\R]}(d;s)$ and $(-1)^k \ W_{[X_\R]}(d-k[S];s)$ have the same sign.
Now the result follows from Theorem \ref{thm:main}.
\end{proof}

\section{Hypothetical connection to tropical refined invariants}\label{sec:refined}
We end this note by pointing out a possible connection of Theorem
\ref{thm:main} to tropical refined invariants of algebraic surfaces,
and refined Severi degrees.

\subsection{Tropical refined invariants and surgeries along real
  Lagrangian spheres}
A  non-degenerate convex polygon  $\Delta\subset \R^2$ with
vertices in $\Z^2$ defines a complex toric surface $X_\Delta$ together
with a complete linear system $d$.
Block and Göttsche proposed in \cite{BlGo14}
to enumerate  irreducible tropical
curves with Newton polygon $\Delta$ and genus $g$ as proposed in \cite{Mik1},
but replacing Mikhalkin's complex multiplicity with its
quantum analog. 
One obtains in
this way a Laurent polynomial in the variable $q$, called
\emph{tropical refined invariant} and denoted by
$G_{X_\Delta}(d,g)$, which does not depend on the configuration of points
chosen to define it \cite{IteMik13}. 
By \cite{Mik1}, the value $G_{X_\Delta}(d,g)(1)$ recovers the number
of complex
irreducible algebraic curves of genus $g$ in $X_\Delta$ realizing the class $d$,
and passing through a generic configuration of
$c_1(X_\Delta)\cdot d -1+g$ points.
Any complex toric surface has a standard real structure induced by the
standard real structure on $(\C^*)^2$, and we denote by
$X_{\Delta,\R}$ the corresponding real toric surface.
It also follows from
\cite{Mik1} that when $X_{\Delta,\R}$ is a 
real unnodal toric del
Pezzo surface (i.e. $\CP^1\times \CP^1$ or $\CP^2$ blown up in  at 
most 3 real points not contained in a line), one has
\[
G_\Delta(0)(-1)=W_{[X_{\Delta,\R}]}(d;0).
\]
In other words,  tropical refined invariants interpolate between
 Gromov-Witten invariants (for $q=1$) and Welschinger
invariants with $s=0$ (for $q=-1$) of $X_{\Delta,\R}$ when both are defined.  
Tropical refined invariants are conjectured to agree with the
$\chi_y$-refinement of Severi degrees introduced in 
\cite{GotShe12}.

In the case of
 a
real unnodal toric del
Pezzo surface $X_{\Delta,\R}$,
Göttsche and Schroeter defined in \cite{GotSch16} some tropical
refined descendant invariants, denoted by $G_{X_\Delta}(d,0;s)$,
interpolating between
some genus 0 tropical descendent invariants and Welschinger invariants 
$W_{[X_{\Delta,\R}]}(d;s)$, and this  for arbitrary values of $s$.
This work has been then  generalized to all genus 0 tropical descendant
invariants by Blechman and Shustin in \cite{BleShu17}, leaving open
the general  interpretation  of the value at $q=-1$ of these new 
tropical
refined descendant invariants.

Despite recent progress, eg.  \cite{Mik15, NPS16,Bou17}, a general
 enumerative interpretation of tropical
 refined invariants unifying their complex and real aspects
 remains unknown. 

\medskip
Nevertheless,
tropical refined invariants
seem to behave nicely under
degenerations of the ambient algebraic surface (or symplectic
4-manifold). In particular, 
a connection of our results to
tropical refined invariants may be suggested by the following observation:
 relative Gromov-Witten invariants of symplectic
4-manifolds satisfy a formula similar to the one from Theorem
\ref{thm:main}.
To make this precise, we need first to introduce some additional notations.

Let $(X,\omega_X)$ be a compact symplectic manifold of dimension 4,
containing a finite union $U= E_1\cup \ldots\cup E_\kappa$
of pairwise disjoint
embedded symplectic spheres
with $[E_i]^2=-2$.
Let also 
$J$ be  an almost complex structure on $X$ tamed by $\omega_{X}$,
 for which all curves $E_1,\ldots,E_\kappa$ are
 $J$-holomorphic.
Given $d \in H_2(X;\Z)$,
let us choose a configuration $\x$ of  $c_1(X)\cdot d-1$ distinct
points in $X\setminus \displaystyle \bigcup_{i=1}^\kappa E_i$. 
We define 
$GW^U_{(X,\omega)}(d)$ as the number of
irreducible $J$-holomorphic rational curves $f:C\to X$ 
 with  $f_*[C]=d$, passing through all points in
 $\x$, and whose image is not contained in
 $\displaystyle \bigcup_{i=1}^\kappa E_i$.
 For a generic choice of $J$,
 this number is finite and does not depend on $\x$, it is called
 a \emph{Gromov-Witten invariant of $(X,\omega)$ relative to $U$}.
 Given an element $E$ of $U$, we define
 $\widehat U=U\setminus \{E\}$. Given two integer numbers $m$ and $k$,
 we define
 \[
 u_{m,k}=(-1)^{k}\left(
\binom{ m +k}{ m} + 
\binom{ m +k-1}{ m} 
\right).
 \]
 \begin{prop}\label{prop:invABV}
   One has
 \[
 GW^{U}_{(X,\omega)}(d)=\sum_{k\ge 0}
 u_{d\cdot [E],k}\
 GW^{\widehat U}_{(X,\omega)}(d-k[E]).
 \]
 In particular if $d\cdot [E]=0$, then one has
  \[
 GW^{U}_{(X,\omega)}(d)= \sum_{k\in \Z}
 (-1)^{k}
 GW^{\widehat U}_{(X,\omega)}(d-k[E]).
 \]
 \end{prop}
 \begin{proof}
 It follows immediately from \cite[Corollary 3.8]{BP14} that
 \[
 GW^{\widehat U}_{(X,\omega)}(d)=\sum_{k\ge 0}
 \left(\begin{array}{c}d\cdot [E]+2k\\ k\end{array}\right)
 GW^U_{(X,\omega)}(d-k[E]).
 \]
 So the first formula follows from \cite[Proposition 3.12]{Bru16}.
 If  $d\cdot [E]=0$, this formula becomes
 \[
 GW^{U}_{(X,\omega)}(d)= GW^{\widehat U}_{(X,\omega)}(d)+ 2\sum_{k\ge 0}
 (-1)^{k}
 GW^{\widehat U}_{(X,\omega)}(d-k[E]).
 \]
Thanks to the equality
 $GW^{U}_{(X,\omega)}(d')=GW^{U}_{(X,\omega)}(d'+(d'\cdot
[E])[E])$ for any $d'\in H_2(X;\Z)$, we obtain the desired result.
 \end{proof}

 The combination of
 Theorem \ref{thm:main} with Proposition \ref{prop:invABV}
 implies the following ``refined flavored'' corollary.

 \begin{cor}\label{cor:refined}  
   Let $X_\R$ be either the quadric hyperboloid $QH$ in $\CP^3$,
or $\CP^2$
blown-up at
 finitely many real points. Let also  
   $U= E_1\cup \ldots\cup E_\kappa$ be a finite collection of
pairwise disjoint
real embedded symplectic spheres
with $[E_i]^2=-2$.
Suppose that we are given a function
$  \Gamma_{X_\R}:H_2(X;\Z)\to \Z[q,q^{-1}]$ such that
for any $d\in H_2(X;\Z)$, one has
\[
\Gamma_{X_\R}(d)(1)=GW_{(X,\omega)}(d) \qquad\mbox{and}\qquad
\Gamma_{X_\R}(d)(-1)=W_{[X_\R]}(d;0).
\]
Let $Y_\R$ be the surgery of $X_\R$ along the real Lagrangian spheres
$E_1,\cdots,E_\kappa$, and define the following function
\[
\begin{array}{cccc}
  \Gamma_{Y_\R}:&H_2(X;\Z)&\longrightarrow& \Z[q,q^{-1}]
  \\ & d&\longmapsto&
  \displaystyle \sum_{k_1,\cdots,k_\kappa \ge 0}
  \left(\prod_{i=1}^\kappa u_{d\cdot [E_i],k_i}  \right)
  \Gamma_{X_\R}(d-k_1[E_1]- \cdots k_\kappa[E_\kappa]).
\end{array}
\]
Then $ \Gamma_{Y_\R}$ satisfies the two following properties: 
\begin{align*}
 &  \forall  d\in H_2(X;\Z), \ 
   \Gamma_{Y_\R}(d)(1)=GW^U_{(X,\omega)}(d); 
   \\& \forall  d\in H_2^{-\tau_Y}(X;\Z), \ \Gamma_{Y_\R}(d)(-1)=W_{Y_\R}(d;0).
\end{align*}
\end{cor}

Following  Corollary \ref{cor:refined},  we discuss in the next two subsections
a refined version of Corollary  \ref{cor:wel}, and of
\cite[Proposition 2.7]{BP14}.

\subsection{A refined Corollary \ref{cor:wel}}

Refined tropical descendant
invariants defined in \cite{GotSch16} can be computed 
using floor diagrams from \cite{BGM}, equipped with refined weights as
indicated in \cite{BleShu17}.
Next proposition, whose proof we omit, is an easy corollary of
this floor diagrammatic computation.
Recall that $QH$ is the real quadric hyperboloid in $\CP^3$.

\begin{prop}\label{prop:refined wel}
  Let  $X_\R$ be either $QH$ or $\CP^2$ blown-up in at most two
  distinct
  real points, and
 let $\widetilde X_\R$ be the blow-up of $X_\R$ at one real point. We denote by $E$ the exceptional divisor.
 Then for any $d\in H_2(X;\Z)$ and $s\in\Z_{\ge 0}$ such that
  $c_1(X)\cdot  d-2s\ge 3$, one has
  \[
  G_{X_\R}(d,0;s+1)=G_{X_\R}(d,0;s) - 2G_{\widetilde X_\R}(d-2E,0;s).
  \]
\end{prop}

By Corollary \ref{cor:refined},
Proposition \ref{prop:refined wel} is clearly a refined version of
Corollary \ref{cor:wel}. In particular, an induction on $s$ provides
an alternative proof
of  \cite[Corollaries 3.14 and 3.29]{GotSch16} (however it does not
explain the tropical invariance of the numbers $G_{X_\R}(d,0;s)$).

\begin{exa}
One computes easily, for example using floor diagrams, that the coefficient of degree
$\frac{(d-1)(d-2)}{2}-1$ of $G_{\CP^2}(d,0;0)$ is $3d+1$. Since this
coefficient is $1$ for $G_{\widetilde{\CP^2}}(d-2E,0;s)$,
we deduce from Proposition \ref{prop:refined wel}
that the coefficient of degree
$\frac{(d-1)(d-2)}{2}-1$ of $G_{\CP^2}(d,0;s)$ is $3d+1-2s$. In
particular when  $s$ is maximal, this coefficient is equal to $2$ or
$3$  depending on the
parity of $d$.
\end{exa}

Since all coefficients of $G_{X_\R}(d,0;s)$ are
positive for an unnodal real del Pezzo toric surface $X_\R$,
Proposition \ref{prop:refined wel} immediately implies the
following.
\begin{cor}\label{cor:refined IKS}
   Let  $X_\R$ be either $QH$ or $\CP^2$ blown-up in at most two
  distinct
  real points. For $d\in H_2(X;\Z)$, we denote by
  $a_{r,s}$ the coefficient of degree $r$ of $G_{X_\R}(d,0;s)$.
  Then one has
  \[
\forall  r\in\left\{1,\cdots , \frac{d^2-c_1(X)\cdot d+2}{2}\right\},\qquad   a_{r,0}\ge a_{r,1}\ge \cdots \ge a_{r,\left[ \frac{c_1(X)\cdot
        d-1}{2}\right]}> 0.
  \]
\end{cor}
Corollary \ref{cor:refined IKS} may be seen as a refined version of
\cite[Conjecture 6]{IKS2}. It is amusing that although this conjecture
has been proven to be wrong in \cite{Wel4} and \cite{Br8}, its refined
version eventually holds.

\subsection{Refined invariants of $\CP^1\times\CP^1$ and $\Sigma_2$}
Recall that the quadric hyperboloid $QH$ can be deformed to the second
Hirzebruch surface $\Sigma_2$ equipped with its toric real structure. 
This can be done algebraically
by degenerating the quadric hyperboloid $QH$ to a
nodal quadric, and by blowing up the node.
The surgery of $X_\R$ along the
expcetional curve in $\Sigma_2$
is, up to deformation, the quadric ellipsoid 
 $QE$ in $\CP^3$. Since both $\CP^1\times \CP^1$ and $\Sigma_2$ are
 toric, Corollary \ref{cor:refined} suggests  the following
 conjecture\footnote{As stated, Corollary \ref{cor:refined} only
   suggests the conjecture for $g=0$. Nevertheless,
   \cite[Corollary 3.17]{Bru16} suggests its extension to any
   genus.} \footnote{This conjecture has been now proved by Bousseau in
   \cite{Bou19}}. 
 \begin{conj}\label{conj:quadric}
   We denote by $\square_{a,b}$ (resp. $\triangle_{a,b}$) the
class in $H_2(\CP^1\times \CP^1;\Z)$ (resp. $\Sigma_2$) defined
by the 
   convex
polygon in  $\R^2$ with 
vertices $(0,0)$, $(a,0)$, $(0,b)$, and $(a,b)$
(resp. $(0,0)$, $(2a+b,0)$, $(0,a)$, and $(b,a)$), see Figure
\ref{fig:NP1}.
\begin{figure}[h!]
\begin{center}
\begin{tabular}{ccc}
  \includegraphics[width=4cm, angle=0]{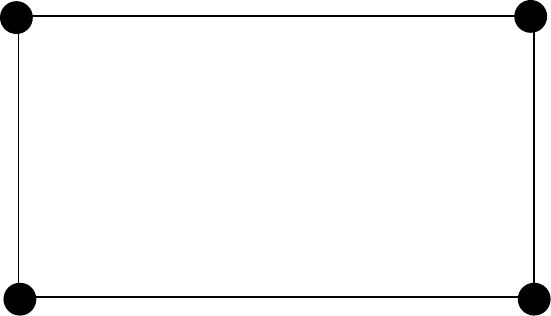}
  \put(-140,0){$(0,0)$}
  \put(-140,60){$(0,b)$}
  \put(2,60){$(a,b)$}
  \put(2,0){$(a,0)$}
& \hspace{18ex}
&\includegraphics[width=4cm, angle=0]{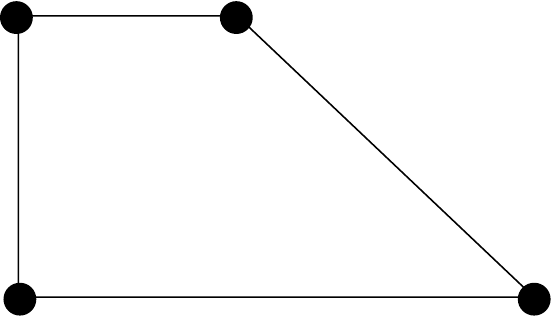}
  \put(-140,0){$(0,0)$}
  \put(-140,60){$(0,a)$}
  \put(-55,60){$(b,a)$}
  \put(2,0){$(2a+b,0)$}
\\ \\a) $\square_{a,b}$
&&b) $\triangle_{a,b}$
\end{tabular}
\end{center}
\caption{}
\label{fig:NP1}
\end{figure}
Then for any integers $a,b,g\ge 0$, the tropical
refined invariants $G_{QH}(\square,g)$ and $G_{\Sigma_2}(\triangle,g)$ satisfy the
following relations:
\[
G_{\Sigma_2}(\triangle_{a,b},g)= \sum_{k\ge 0}
u_{b,k} G_{QH}(\square_{a+b+k,a-k},g).
\]
 \end{conj}
 Note that this can be seen as
 a refined version of \cite[Proposition 2.7]{BP14}:  denoting by $h$ the
 hyperplane section class of the quadric ellipsoid $QE$, one has
 \[
 \forall a\ge 0, \ G_{\Sigma_2}(\triangle_{a,0},0)(-1)=  W_{QE}(ah ;0).
 \]
In the next examples, we check that Conjecture \ref{conj:quadric}
holds in a few cases.
\begin{exa}
  The cases $a=1$ or $g=(a-1)(a-1+b)$
  hold trivially. 
\end{exa}

The value of all refined invariants needed in the next examples are
provided in Appendix \ref{app:computation}.

\begin{exa}
  In the case $(a,b)=(2,0)$ and $g=0$, one has
  \[
  G_{\Sigma_2}(\triangle_{2,0},0;s)= G_{QH}(\square_{2,2},0;s) -2G_{QH}(\square_{3,1},0;s).
  \]
\end{exa}

\begin{exa}
  The case $(a,b)=(2,2)$ gives
  \begin{align*}
    &g=2:& &  G_{\Sigma_2}(\triangle_{2,2},2)= G_{QH}(\square_{4,2},2);
    \\    & g=1: &&   G_{\Sigma_2}(\triangle_{2,2},1)=  G_{QH}(\square_{4,2},1);
    \\     & g=0 :&&
    G_{\Sigma_2}(\triangle_{2,2},0;s)= G_{QH}(\square_{4,2},0;s)
    - 4  G_{QH}(\square_{5,1},0;s).
 \end{align*}
  
\end{exa}

\begin{exa}
  A more interesting case is given by $(a,b)=(3,0)$ and $g\in\{0,1,2,3\}$.
  We obtain the following
  \begin{align*}
    & g=3:&&
 G_{\Sigma_2}(\triangle_{3,0},3)= G_{QH}(\square_{3,3},3) - 2G_{QH}(\square_{4,2},3).
\\ &g=2: &&
  G_{\Sigma_2}(\triangle_{3,0},2)= G_{QH}(\square_{3,3},2) -2G_{QH}(\square_{4,2},2).
 \\ &g=1: &&
  G_{\Sigma_2}(\triangle_{3,0},1)= G_{QH}(\square_{3,3},1) -2G_{QH}(\square_{4,2},1).
 \\ &g=0: &&
   G_{\Sigma_2}(\triangle_{3,0},0;s)= G_{QH}(\square_{3,3},0;s) -2G_{QH}(\square_{4,2},0;s)+
   2G_{QH}(\square_{5,1},0;s).
 \end{align*}
\end{exa}

A refined version of the strategy used in \cite{Br18} may lead to
a combinatorial proof of Conjecture \ref{conj:quadric}, however it is not
clear that such a combinatorial proof will be geometrically
meaningful.

\appendix
\section{Some computations of tropical refined invariants}\label{app:computation}

In this appendix we provide a few values of tropical refined invariants
of tropical projective plane and Hirzebruch surfaces of small degree. 
All non-trivial computations of absolute refined invariants have been done using floor diagrams
\cite{Br6b,BlGo14,BIMS15}. The computations of tropical refined descendant
invariants have been done
using floor diagrams from \cite{BGM} equipped with refined weights as
indicated in \cite{BleShu17}.

\subsection{$\T P^1\times \T P^1$}
$\ $

\noindent Here we give the values of $G_{QH}(\square_{a,b},g)$  for a few
 $a$ and $b$.
Note that $G_{QH}(\square_{a,b})=G_{QH}(\square_{b,a})$

\begin{itemize}
\item {\bf $a=1$:}
   \begin{align*}
&&& 
  G_{QH}(\square_{1,b},0)=1 &\hfill
\end{align*}

\item {\bf $(a,b)=(2,0)$:}
   \begin{align*}
         & g=1 :&&
    G_{QH}(\square_{2,2},1)=1  & \hfill 
    \\     & g=0 :&&
  G_{QH}(\square_{2,2},0;s)=q^{-1}+ (10-2s) + q
\end{align*}

\item {\bf $(a,b)=(2,4)$:}
  \begin{align*}
    &g=3:& &   G_{QH}(\square_{2,4},3)= 1 &\hfill
\\    &g=2:& &   G_{QH}(\square_{2,4},2)= 3q^{-1} +22+3q
    \\    & g=1: &&  G_{QH}(\square_{2,4},1) = 3q^{-2} +36q^{-1} +162+
     36q +3q^2
    \\     & g=0 :&&
     G_{QH}(\square_{2,4},0;0)= q^{-3}+ 14q^{-2}+ 95q^{-1} +420+ 95q +
     14q^{2}+q^{3}
    \\     & &&
     G_{QH}(\square_{2,4},0;1)= q^{-3}+ 12q^{-2}+ 71q^{-1} +280+ 71q +
     12q^{2}+q^{3}
    \\     & &&
     G_{QH}(\square_{2,4},0;2)= q^{-3}+ 10q^{-2}+ 51q^{-1} +180+ 51q +
     10q^{2}+q^{3}
     \\     & &&
     G_{QH}(\square_{2,4},0;3)= q^{-3}+ 8q^{-2}+ 35q^{-1} +112+ 35q +
     8q^{2}+q^{3}
    \\     & &&
     G_{QH}(\square_{2,4},0;4)= q^{-3}+ 6q^{-2}+ 23q^{-1} +68+ 23q +
     6q^{2}+q^{3}
    \\     & &&
     G_{QH}(\square_{2,4},0;5)= q^{-3}+ 4q^{-2}+ 15q^{-1} +40+ 15q +
     4q^{2}+q^{3}
\end{align*}

\item {\bf $(a,b)=(3,3)$:}

  \begin{align*}
      &g=4:& &  G_{QH}(\square_{3,3},4)=1
  \\  &g=3:& &  G_{QH}(\square_{3,3},3)=4q^{-1} +26+ 4q
  \\  &g=2:& &  G_{QH}(\square_{3,3},2)=6q^{-2}+64q^{-1}+256+ 64q+6q^2
  \\    & g=1: &&   G_{QH}(\square_{3,3},1)=4q^{-3}+52q^{-2}+332q^{-1}+1144+332q+52q^2+4q^3
    \\     & g=0 :&&
    G_{QH}(\square_{3,3},0;0)=
    q^{-4}+14q^{-3}+109q^{-2}+592q^{-1}+2078+592q+109q^2+14q^3+q^4
    \\     & &&
     G_{QH}(\square_{3,3},0;1)=q^{-4}+
     12q^{-3}+83q^{-2}+404q^{-1}+1270+404q+83q^2+12q^3+q^4
    \\     & &&
     G_{QH}(\square_{3,3},0;2)= q^{-4}+
     10^{-3}+61q^{-2}+264^{-1}+742+264q+61q^2+10q^3+q^4
     \\     & &&
     G_{QH}(\square_{3,3},0;3)= q^{-4}+
    8q^{-3}+43q^{-2}+164q^{-1}+ 414+164q+43q^2+8q^3+q^4
    \\     & &&
     G_{QH}(\square_{3,3},0;4)= q^{-4}+
    6q^{-3}+29q^{-2}+96q^{-1}+ 222+96q+29q^2+6q^3+q^4
    \\     & &&
     G_{QH}(\square_{3,3},0;5)= q^{-4}+
     4q^{-3}+19q^{-2}+52q^{-1}+118+52q+19q^2+4q^3+q^4
 \end{align*}
\end{itemize}

\subsection{$\T \Sigma_2$}
$\ $

\noindent Here we give the values of $G_{\Sigma_2}(\triangle_{a,b},g)$  for a few  $a$ and $b$.

\begin{itemize}
\item {\bf $a=1$:}
  \[
  G_{\Sigma_2}(\triangle_{1,b},0)=1
  \]

\item {\bf $(a,b)=(2,0)$:}
  \begin{align*}
         & g=1 :&&
    G_{\Sigma_2}(\triangle_{2,0},1)=1  & \hfill 
    \\     & g=0 :&&
   G_{\Sigma_2}(\triangle_{2,0},0;s)=q^{-1}+ (8-2s) + q
  \end{align*}

\item {\bf $(a,b)=(2,2)$:}
  \begin{align*}
    &g=3:& &  G_{\Sigma_2}(\triangle_{2,2},3)= 1 & \hfill 
\\    &g=2:& &  G_{\Sigma_2}(\triangle_{2,2},2)= 3q^{-1} +22+3q
    \\    & g=1: &&   G_{\Sigma_2}(\triangle_{2,2},1)= 3q^{-2} +36q^{-1} +162+
     36q +3q^2
    \\     & g=0 :&&
    G_{\Sigma_2}(\triangle_{2,2},0;0)= q^{-3}+ 14q^{-2}+ 95q^{-1} +416+ 95q +
     14q^{2}+q^{3}
    \\     & &&
     G_{\Sigma_2}(\triangle_{2,2},0;1)= q^{-3}+ 12q^{-2}+ 71q^{-1} +276+ 71q +
     12q^{2}+q^{3}
    \\     & &&
     G_{\Sigma_2}(\triangle_{2,2},0;2)= q^{-3}+ 10q^{-2}+ 51q^{-1} +176+ 51q +
     10q^{2}+q^{3}
     \\     & &&
     G_{\Sigma_2}(\triangle_{2,2},0;3)= q^{-3}+ 8q^{-2}+ 35q^{-1} +108+ 35q +
     8q^{2}+q^{3}
    \\     & &&
     G_{\Sigma_2}(\triangle_{2,2},0;4)= q^{-3}+ 6q^{-2}+ 23q^{-1} +64+ 23q +
     6q^{2}+q^{3}
    \\     & &&
     G_{\Sigma_2}(\triangle_{2,2},0;5)= q^{-3}+ 4q^{-2}+ 15q^{-1} +36+ 15q +
     4q^{2}+q^{3}
 \end{align*}
  
\item {\bf $(a,b)=(3,0)$:}

  \begin{align*}
      &g=4:& &  G_{\Sigma_2}(\triangle_{3,0},4)=1
  \\  &g=3:& &  G_{\Sigma_2}(\triangle_{3,0},3)=4q^{-1} +24+ 4q
 \\  &g=2:& &  G_{\Sigma_2}(\triangle_{3,0},2)=6q^{-2} +58q^{-1} +212+
     58q +6q^2
    \\    & g=1: &&   G_{\Sigma_2}(\triangle_{3,0},1)=4q^{-3}+ 46q^{-2} +260q^{-1} +820+
     260q +46q^2 + 4q^3
    \\     & g=0 :&&
    G_{\Sigma_2}(\triangle_{3,0},0;0)= q^{-4}+ 12q^{-3}+ 81q^{-2} +402q^{-1} +1240+
     402q +81q^2 + 12q^3 +q^4
    \\     & &&
     G_{\Sigma_2}(\triangle_{3,0},0;1)=q^{-4}+ 10q^{-3}+ 59q^{-2} +262q^{-1} +712+
     262q +59q^2 + 10q^3 +q^4
    \\     & &&
     G_{\Sigma_2}(\triangle_{3,0},0;2)= q^{-4}+ 8q^{-3}+ 41q^{-2} +162q^{-1} +384+
     162q +41q^2 + 8q^3 +q^4
     \\     & &&
     G_{\Sigma_2}(\triangle_{3,0},0;3)= q^{-4}+ 6q^{-3}+ 27q^{-2} +94q^{-1} +192+
     94q +27q^2 + 6q^3 +q^4
    \\     & &&
     G_{\Sigma_2}(\triangle_{3,0},0;4)= q^{-4}+ 4q^{-3}+ 17q^{-2} +50q^{-1} +88+
     50q +17q^2 + 4q^3 +q^4
    \\     & &&
     G_{\Sigma_2}(\triangle_{3,0},0;5)= q^{-4}+ 2q^{-3}+ 11q^{-2} +22q^{-1} +40+
     22q +11q^2 + 2q^3 +q^4
 \end{align*}
\end{itemize}

\bibliographystyle{alpha}
\bibliography{../../Biblio.bib}

\begin{thebibliography}{ABLdM11}

\bibitem[ABLdM11]{Br8}
A.~Arroyo, E.~Brugall\'e, and L.~L\'opez~de Medrano.
\newblock Recursive formula for {W}elschinger invariants.
\newblock {\em Int Math Res Notices}, 5:1107--1134, 2011.

\bibitem[BG16]{BlGo14}
F.~Block and L.~G{\"o}ttsche.
\newblock Refined curve counting with tropical geometry.
\newblock {\em Compos. Math.}, 152(1):115--151, 2016.

\bibitem[BGM12]{BGM}
F.~Block, A.~Gathmann, and H.~Markwig.
\newblock Psi-floor diagrams and a {C}aporaso-{H}arris type recursion.
\newblock {\em Israel J. Math.}, 191(1):405--449, 2012.

\bibitem[BIMS15]{BIMS15}
E.~Brugall{\'e}, I.~Itenberg, G.~Mikhalkin, and K.~Shaw.
\newblock Brief introduction to tropical geometry.
\newblock In {\em Proceedings of the {G}\"okova {G}eometry-{T}opology
  {C}onference 2014}, pages 1--75. G\"okova Geometry/Topology Conference (GGT),
  G\"okova, 2015.

\bibitem[BM08]{Br6b}
E.~Brugall\'e and G.~Mikhalkin.
\newblock Floor decompositions of tropical curves : the planar case.
\newblock {\em Proceedings of 15th {G}\"okova {G}eometry-{T}opology
  Conference}, pages 64--90, 2008.

\bibitem[BM16]{Br18}
E.~Brugall\'{e} and H.~Markwig.
\newblock Deformation and tropical {H}irzebruch surfaces and enumerative
  geometry.
\newblock {\em J. Algebraic Geom.}, 25(4):633--702, 2016.

\bibitem[Bou17]{Bou17}
P~Bousseau.
\newblock Tropical refined curve counting from higher genera and lambda
  classes.
\newblock arXiv:1706.0776, 2017.

\bibitem[BP13]{Br20}
E.~Brugall{\'e} and N.~Puignau.
\newblock Behavior of {W}elschinger invariants under {M}orse simplifications.
\newblock {\em Rend. Semin. Mat. Univ. Padova}, 130:147--153, 2013.

\bibitem[BP15]{BP14}
E.~Brugall\'e and N.~Puignau.
\newblock On {W}elschinger invariants of symplectic 4-manifolds.
\newblock {\em Comment. Math. Helv.}, 90(4):905--938, 2015.

\bibitem[Bru15]{Bru14}
E.~Brugall{\'e}.
\newblock Floor diagrams relative to a conic, and {GW}--{W} invariants of {D}el
  {P}ezzo surfaces.
\newblock {\em Adv. Math.}, 279:438--500, 2015.

\bibitem[Bru18]{Bru16}
E.~Brugall\'e.
\newblock Surgery of real symplectic fourfolds and welschinger invariants.
\newblock {\em Journal of Singularities}, 17:267--294, 2018.

\bibitem[BS16]{BleShu17}
L.~Blechman and E.~Shustin.
\newblock Refined descendant invariants of toric surfaces.
\newblock https://arxiv.org/abs/1602.06471, 2016.

\bibitem[Che18]{Che18}
X~Chen.
\newblock Solomon's relations for {W}elshinger's invariants: Examples.
\newblock arXiv:1809.08938, 2018.

\bibitem[DH18]{DiHu18}
Y.~Ding and J.~Hu.
\newblock Welschinger invariants of blow-ups of symplectic 4-manifolds.
\newblock {\em Rocky Mountain J. Math.}, Volume 48(4):1105--1144, 2018.

\bibitem[DK00]{DK}
A.~I. Degtyarev and V.~M. Kharlamov.
\newblock Topological properties of real algebraic varieties: {R}okhlin's way.
\newblock {\em Russian Math. Surveys}, 55(4):735--814, 2000.

\bibitem[DK02]{DegKha02}
A.~Degtyarev and V.~Kharlamov.
\newblock Real rational surfaces are quasi-simple.
\newblock {\em J. Reine Angew. Math.}, 551:87--99, 2002.

\bibitem[Geo16]{Geor13}
P.~Georgieva.
\newblock Open {G}romov-{W}itten disk invariants in the presence of an
  anti-symplectic involution.
\newblock {\em Adv. Math.}, 301:116--160, 2016.

\bibitem[GS14]{GotShe12}
L.~G{\"o}ttsche and V.~Shende.
\newblock Refined curve counting on complex surfaces.
\newblock {\em Geom. Topol.}, 18(4):2245--2307, 2014.

\bibitem[GS16]{GotSch16}
L.~G{\"o}ttsche and F.~Schroeter.
\newblock Refined broccoli invariants.
\newblock arXiv:1606.09631, 2016.

\bibitem[GZ15]{GeoZin15}
P.~Georgieva and A.~Zinger.
\newblock Real {G}romov-{W}itten theory in all genera and real enumerative
  geometry: Construction.
\newblock arXiv:1504.06617, 2015.

\bibitem[HS12]{HorSol12}
A.~Horev and J.~Solomon.
\newblock The open {G}romov-{W}itten-{W}elschinger theory of blowups of the
  projective plane.
\newblock arXiv:1210.4034, 2012.

\bibitem[IKS04]{IKS2}
I.~Itenberg, V.~Kharlamov, and E.~Shustin.
\newblock Logarithmic equivalence of {W}elschinger and {G}romov-{W}itten
  invariants.
\newblock {\em Uspehi Mat. Nauk}, 59(6):85--110, 2004.
\newblock (in Russian). English version: Russian Math. Surveys 59 (2004), no.
  6, 1093-1116.

\bibitem[IKS13]{IKS11}
I.~Itenberg, V.~Kharlamov, and E.~Shustin.
\newblock Welschinger invariants of real del {P}ezzo surfaces of degree
  {$\geq3$}.
\newblock {\em Math. Ann.}, 355(3):849--878, 2013.

\bibitem[IKS15]{IKS13}
I.~Itenberg, V.~Kharlamov, and E.~Shustin.
\newblock Welschinger invariants of real del {P}ezzo surfaces of degree {$\geq
  2$}.
\newblock {\em Internat. J. Math.}, 26(8):1550060, 63, 2015.

\bibitem[IKS17]{IKS14}
I.~Itenberg, V.~Kharlamov, and E.~Shustin.
\newblock Welschinger invariants revisited.
\newblock In {\em Analysis meets geometry}, Trends Math., pages 239--260.
  Birkh\"auser/Springer, Cham, 2017.

\bibitem[IM13]{IteMik13}
I.~Itenberg and G.~Mikhalkin.
\newblock On {B}lock-{G}\"ottsche multiplicities for planar tropical curves.
\newblock {\em Int. Math. Res. Not. IMRN}, (23):5289--5320, 2013.

\bibitem[IP04]{IP}
E.-N Ionel and T.~H. Parker.
\newblock The symplectic sum formula for {G}romov-{W}itten invariants.
\newblock {\em Ann. of Math.}, 159(2):935--1025, 2004.

\bibitem[Li02]{Li02}
J.~Li.
\newblock A degeneration formula of {GW}-invariants.
\newblock {\em J. Differential Geom.}, 60(2):199--293, 2002.

\bibitem[Li04]{Li04}
J.~Li.
\newblock Lecture notes on relative {GW}-invariants.
\newblock In {\em Intersection theory and moduli}, ICTP Lect. Notes, XIX, pages
  41--96 (electronic). Abdus Salam Int. Cent. Theoret. Phys., Trieste, 2004.

\bibitem[LR01]{LiRu01}
A.~Li and Y.~Ruan.
\newblock Symplectic surgery and {G}romov-{W}itten invariants of {C}alabi-{Y}au
  3-folds.
\newblock {\em Invent. Math.}, 145(1):151--218, 2001.

\bibitem[Mik05]{Mik1}
G.~Mikhalkin.
\newblock {Enumerative tropical algebraic geometry in $\mathbb R^2$}.
\newblock {\em J. Amer. Math. Soc.}, 18(2):313--377, 2005.

\bibitem[Mik17]{Mik15}
G.~Mikhalkin.
\newblock Quantum indices of real plane curves and refined enumerative
  geometry.
\newblock {\em Acta Math.}, 219(1):135--180, 2017.

\bibitem[NPS18]{NPS16}
J.~Nicaise, S.~Payne, and F.~Schroeter.
\newblock Tropical refined curve counting via motivic integration.
\newblock {\em Geom. Topol.}, 22(6):3175--3234, 2018.

\bibitem[Shu14]{Shu14}
E.~Shustin.
\newblock On higher genus {W}elschinger invariants of del pezzo surfaces.
\newblock {\em Intern. Math. Res. Notices}, 2014.
\newblock doi: 10.1093/imrn/rnu148.

\bibitem[TZ14]{TehZin14}
M.F. Tehrani and A.~Zinger.
\newblock On symplectic sum formulas in {G}romov-{W}itten theory.
\newblock arXiv:1404.1898, 2014.

\bibitem[Wel05]{Wel1}
J.~Y. Welschinger.
\newblock Invariants of real symplectic 4-manifolds and lower bounds in real
  enumerative geometry.
\newblock {\em Invent. Math.}, 162(1):195--234, 2005.

\bibitem[Wel07]{Wel4}
J.~Y. Welschinger.
\newblock Optimalit\'e, congruences et calculs d'invariants des vari\'et\'es
  symplectiques r\'eelles de dimension quatre.
\newblock arXiv:0707.4317, 2007.

\bibitem[Wel15]{Wel11}
J.-Y. Welschinger.
\newblock Open {G}romov-{W}itten invariants in dimension four.
\newblock {\em J. Symplectic Geom.}, 13(4):1075--1100, 2015.

\end{thebibliography}

\end{document}